\documentclass[]{amsart}
\usepackage{yufei}
\newcommand{\Z}[0]{\mathbb{Z}}
\newcommand{\nl}[0]{\newline}
\newcommand{\ssthat}[2]{\lbrace #1 \;\big|\; #2 \rbrace}
\usepackage{amssymb}
\usepackage{amsthm}
\usepackage{amsmath}
\usepackage{mathrsfs}
\usepackage{amsbsy}
\usepackage[all]{xy}
\usepackage{bm}
\usepackage{hyperref}
\usepackage{tikz}
\usepackage{array}
\usepackage{float}
\usepackage{enumerate}
\usepackage{xcolor}
\usepackage{hhline}
\allowdisplaybreaks
\usepackage[noadjust]{cite}

\usepackage[labelfont=bf,justification=raggedright,singlelinecheck=false]{caption}
\usepackage{tabu}
\usepackage{diagbox}
\newenvironment{enumerate*}%
{\begin{enumerate}[(I)]%
		\setlength{\itemsep}{10pt}%
		\setlength{\parskip}{0pt}}%
	{\end{enumerate}}

\begin{document}
	
	\title[]{Cylindrical Hitomezashi Patterns}
	\subjclass[2020]{}
	
	\author[Edwin Xie]{Edwin Xie}
	\address[]{Department of Mathematics, Stanford University, Stanford, CA 95014, USA}
	\email{ex26@stanford.edu}
	
	\maketitle
	
	\begin{abstract}
		Pete discovered a strong combinatorial description of hitomezashi loops via a bijection to pairs of Dyck paths of the same height.
		Our main theorem provides an analogous description of hitomezashi loops of nonzero homology class on certain cylindrical hitomezashi patterns. In the process, we complete some of Ren and Zhang's results on the count and possible homology classes of toroidal hitomezashi loops.
	\end{abstract}
	
	\section{Introduction}\label{sec:intro}
	
		Hitomezashi, a Japanese hand embroidery style, has recently attracted attention for its interesting combinatorial and geometric properties. After the Youtube channel \emph{Numberphile} popularized the mathematical definition of hitomezashi patterns in \cite{Numberphile}, Defant and Kravitz \cite{DefantKravitz} established fundamental structural properties of loops present in hitomezashi patterns (called ``hitomezashi loops''), chiefly that the length of such a loop is congruent to 4 modulo 8 and the area it encloses is congruent to 1 modulo 4. Hitomezashi patterns were also previously studied by Pete ~\cite{Pete} under the name \emph{corner percolation}. One of his main results was a bijection between hitomezashi loops and pairs of Dyck paths of the same height arising from the binary encodings of the horizontal and vertical edges in the underlying hitomezashi pattern.

        %The *only* reason this is interesting is it allows for loops with nontrivial homology class 
		Ren and Zhang \cite{RenZhang} extended the study of hitomezashi patterns to the torus, establishing modular-theoretic results on the lengths of loops with nontrivial homology class, analogous to the results of Defant and Kravitz in \cite{DefantKravitz}. They also gave a partial classification of the homology classes that occur and, in certain cases, explicitly enumerated the corresponding loops. In the special case where the horizontal and vertical strings coincide, they applied a clever knot-theoretic argument to derive further constraints on the number of nontrivial loops.

		This paper considers hitomezashi patterns on a cylinder. Analogous to Pete's \cite{Pete} result in the planar case, we establish a bijection relating the appearances of hitomezashi loops of nontrivial homology class to certain subsequences of the encoding strings. We then apply these results to toroidal hitomezashi patterns. Extending the work of Ren and Zhang \cite{RenZhang}, we completely classify loops with nontrivial homology and use this to count loops with given homology.
	    
		First, we recall the original ``unoriented" formulation of hitomezashi patterns; in the interest of generalizing to the cylinder and torus, we later introduce the ``oriented" version due to Ren and Zhang.
		
		%DEFINE CLOTH
		\begin{definition}
			\label{defn:cloth}
			Let $G, H \in \lbrace \mathbb{Z}\rbrace \cup \ssthat{ \mathbb{Z}/N\Z}{N\in \Z_+}$. The corresponding grid graph, denoted $\Cloth_{G\times H}$, is the graph on $G\times H$ with $(i, j)$ adjacent to $(i, j\pm 1)$ and $(i \pm 1, j)$.
		\end{definition}
		
		\begin{definition}
			\label{defn:unoriented}
			Let $x, y\in \lbrace -1, 1\rbrace^\Z$. The corresponding \textbf{unoriented planar hitomezashi pattern} is the subgraph of $\Cloth_{\Z\times \Z}$ with edge set
				\[\lbrace \lbrace (i, j), (i+x_j, j)\rbrace: i \equiv 0 \bmod{2}\rbrace \cup \lbrace \lbrace (i, j), (i, j+y_i) \rbrace: j \equiv 0 \bmod{2}\rbrace\]
			An \textbf{unoriented planar hitomezashi path} is a path in an unoriented planar hitomezashi pattern. An \textbf{unoriented planar hitomezashi loop} is a cycle in an unoriented planar hitomezashi pattern. See Figure \ref{fig:mydiagram}.
		\end{definition}
\par\noindent
%INSERT PICTURE!

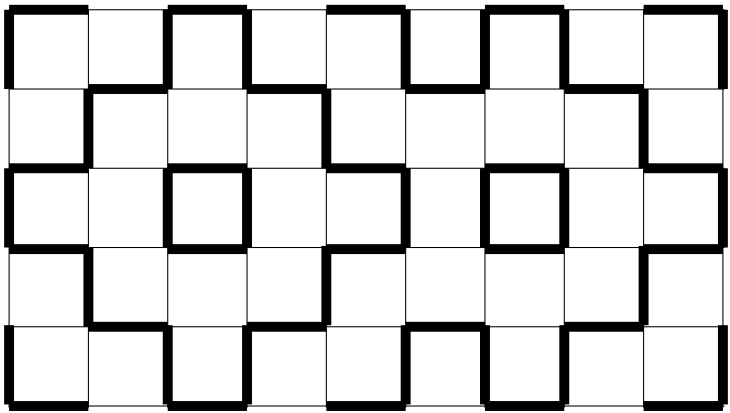
\begin{figure}[ht]
    \centering

\begin{tikzpicture}[x=0.75pt,y=0.75pt,yscale=-0.8,xscale=0.8]

%Shape: Grid [id:dp7842182217403137] 
\draw  [draw opacity=0] (106,32) -- (557,32) -- (557,283) -- (106,283) -- cycle ; \draw   (106,32) -- (106,283)(156,32) -- (156,283)(206,32) -- (206,283)(256,32) -- (256,283)(306,32) -- (306,283)(356,32) -- (356,283)(406,32) -- (406,283)(456,32) -- (456,283)(506,32) -- (506,283)(556,32) -- (556,283) ; \draw   (106,32) -- (557,32)(106,82) -- (557,82)(106,132) -- (557,132)(106,182) -- (557,182)(106,232) -- (557,232)(106,282) -- (557,282) ; \draw    ;
%Straight Lines [id:da13018657607334538] 
\draw [line width=3.75]    (106,32) -- (128,32) -- (156,32) ;
%Straight Lines [id:da6958771178980943] 
\draw [line width=3.75]    (206,32) -- (256,32) ;
%Straight Lines [id:da9607428789706747] 
\draw [line width=3.75]    (306,32) -- (328,32) -- (356,32) ;
%Straight Lines [id:da0529682982474442] 
\draw [line width=3.75]    (406,32) -- (428,32) -- (456,32) ;
%Straight Lines [id:da3308746427212925] 
\draw [line width=3.75]    (506,32) -- (528,32) -- (556,32) ;
%Straight Lines [id:da4757223282717504] 
\draw [line width=3.75]    (106,132) -- (128,132) -- (156,132) ;
%Straight Lines [id:da8912503464747357] 
\draw [line width=3.75]    (206,132) -- (256,132) ;
%Straight Lines [id:da44652234748182273] 
\draw [line width=3.75]    (306,132) -- (328,132) -- (356,132) ;
%Straight Lines [id:da6172701159422576] 
\draw [line width=3.75]    (406,132) -- (428,132) -- (456,132) ;
%Straight Lines [id:da8378422531359442] 
\draw [line width=3.75]    (506,132) -- (528,132) -- (556,132) ;
%Straight Lines [id:da8868146741293914] 
\draw [line width=3.75]    (106,183) -- (128,183) -- (156,183) ;
%Straight Lines [id:da8151506063717477] 
\draw [line width=3.75]    (206,183) -- (256,183) ;
%Straight Lines [id:da9200848155329082] 
\draw [line width=3.75]    (306,183) -- (328,183) -- (356,183) ;
%Straight Lines [id:da03432645831818837] 
\draw [line width=3.75]    (406,183) -- (428,183) -- (456,183) ;
%Straight Lines [id:da009433731977862347] 
\draw [line width=3.75]    (506,183) -- (528,183) -- (556,183) ;
%Straight Lines [id:da08407546540249755] 
\draw [line width=3.75]    (106,282) -- (128,282) -- (156,282) ;
%Straight Lines [id:da9923411504821462] 
\draw [line width=3.75]    (206,282) -- (256,282) ;
%Straight Lines [id:da8077980453459073] 
\draw [line width=3.75]    (306,282) -- (328,282) -- (356,282) ;
%Straight Lines [id:da9508210139625979] 
\draw [line width=3.75]    (406,282) -- (428,282) -- (456,282) ;
%Straight Lines [id:da6908794108307761] 
\draw [line width=3.75]    (506,282) -- (528,282) -- (556,282) ;
%Straight Lines [id:da1662027133718007] 
\draw [line width=3.75]    (156,82) -- (206,82) ;
%Straight Lines [id:da9665713910723228] 
\draw [line width=3.75]    (256,82) -- (278,82) -- (306,82) ;
%Straight Lines [id:da0657506482923138] 
\draw [line width=3.75]    (356,82) -- (378,82) -- (406,82) ;
%Straight Lines [id:da43152875167454186] 
\draw [line width=3.75]    (456,82) -- (478,82) -- (506,82) ;
%Straight Lines [id:da22936013532083166] 
\draw [line width=3.75]    (156,232) -- (206,232) ;
%Straight Lines [id:da9601995502607301] 
\draw [line width=3.75]    (256,232) -- (278,232) -- (306,232) ;
%Straight Lines [id:da9016039565742069] 
\draw [line width=3.75]    (356,232) -- (378,232) -- (406,232) ;
%Straight Lines [id:da30969755178460945] 
\draw [line width=3.75]    (456,232) -- (478,232) -- (506,232) ;
%Straight Lines [id:da7303554481189465] 
\draw [line width=3.75]    (106,32) -- (106,82) ;
%Straight Lines [id:da6973829863940324] 
\draw [line width=3.75]    (106,132) -- (106,182) ;
%Straight Lines [id:da7268975611856558] 
\draw [line width=3.75]    (106,231) -- (106,253) -- (106,281) ;

%Straight Lines [id:da0947312156157818] 
\draw [line width=3.75]    (256,32) -- (256,82) ;
%Straight Lines [id:da38453406909078214] 
\draw [line width=3.75]    (256,132) -- (256,182) ;
%Straight Lines [id:da4518990561464884] 
\draw [line width=3.75]    (256,231) -- (256,253) -- (256,281) ;

%Straight Lines [id:da4602605396982631] 
\draw [line width=3.75]    (456,32) -- (456,82) ;
%Straight Lines [id:da848805611738594] 
\draw [line width=3.75]    (456,132) -- (456,182) ;
%Straight Lines [id:da11774238326502795] 
\draw [line width=3.75]    (456,231) -- (456,253) -- (456,281) ;

%Straight Lines [id:da4716071165697522] 
\draw [line width=3.75]    (406,32) -- (406,82) ;
%Straight Lines [id:da9736313457707726] 
\draw [line width=3.75]    (406,132) -- (406,182) ;
%Straight Lines [id:da34956977934415323] 
\draw [line width=3.75]    (406,231) -- (406,253) -- (406,281) ;

%Straight Lines [id:da39053347798350013] 
\draw [line width=3.75]    (206,32) -- (206,82) ;
%Straight Lines [id:da6408932303680253] 
\draw [line width=3.75]    (206,132) -- (206,182) ;
%Straight Lines [id:da770210505825982] 
\draw [line width=3.75]    (206,231) -- (206,253) -- (206,281) ;

%Straight Lines [id:da5008308628695566] 
\draw [line width=3.75]    (356,32) -- (356,82) ;
%Straight Lines [id:da8663137673927837] 
\draw [line width=3.75]    (356,132) -- (356,182) ;
%Straight Lines [id:da7014102854291173] 
\draw [line width=3.75]    (356,231) -- (356,253) -- (356,281) ;

%Straight Lines [id:da19266369776023895] 
\draw [line width=3.75]    (556,32) -- (556,82) ;
%Straight Lines [id:da7535001214324111] 
\draw [line width=3.75]    (556,132) -- (556,182) ;
%Straight Lines [id:da5995823632861745] 
\draw [line width=3.75]    (556,231) -- (556,253) -- (556,281) ;

%Straight Lines [id:da9519737099350307] 
\draw [line width=3.75]    (156,82) -- (156,132) ;
%Straight Lines [id:da13178249816952037] 
\draw [line width=3.75]    (156,182) -- (156,232) ;
%Straight Lines [id:da08694033670817691] 
\draw [line width=3.75]    (306,82) -- (306,132) ;
%Straight Lines [id:da39675049802057716] 
\draw [line width=3.75]    (306,181) -- (306,203) -- (306,231) ;
%Straight Lines [id:da029670322752359812] 
\draw [line width=3.75]    (506,82) -- (506,132) ;
%Straight Lines [id:da3922947795873717] 
\draw [line width=3.75]    (506,181) -- (506,203) -- (506,231) ;

\end{tikzpicture}
\caption{A section of an unoriented hitomezashi pattern.}
    \label{fig:mydiagram}
\end{figure}

	Note that each vertex in the hitomezashi pattern has exactly two neighbors. Thus, each connected component of an unoriented planar hitomezashi pattern is either an infinite unoriented planar hitomezashi path or an unoriented planar hitomezashi loop.
	
	Attempting to define an unoriented hitomezashi pattern analogously on the grid $\Cloth_{\Z/M\Z \times \Z/N\Z}$ fails if either of $M$ or $N$ are odd, since the notions $i \equiv 0 \bmod {2}, j \equiv 0 \bmod{2}$ no longer make sense. In the interest of studying hitomezashi patterns on toroidal grids $\Z/M\Z \times \Z/N\Z$, Ren and Zhang introduced a natural generalization of hitomezashi patterns using oriented edges.
	
	\begin{definition}
		\label{defn: oriented}
		Let $x \in \lbrace -1, 1\rbrace^H, y \in \lbrace -1, 1\rbrace^G$. The corresponding \textbf{oriented hitomezashi pattern}, denoted $\Cloth_{G\times H}(x,y)$, is defined as the following choice of orientation of $\Cloth_{G\times H}$. Orient $(i, j) \rightarrow (i+1, j)$ if $x_j = 1$, and orient $(i, j) \leftarrow (i+1, j)$ if $x_j = -1$. Symmetrically, orient $(i, j) \rightarrow (i, j+1)$ if $y_i = 1$, and orient $(i, j) \rightarrow (i, j+1)$ if $y_i = -1$.
		An \textbf{oriented hitomezashi loop} is a circuit in an oriented hitomezashi pattern that alternates between horizontal and vertical edges. An \textbf{oriented hitomezashi path} is a path in an oriented hitomezashi pattern that alternates between horizontal and vertical edges.
	\end{definition}

\begin{figure}[ht]
    \centering
  \includegraphics[height=8cm]{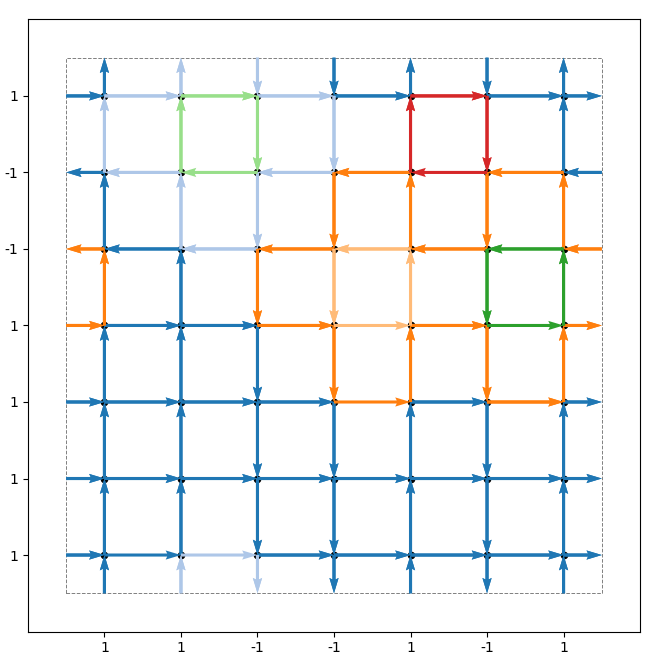}
    \caption{A toroidal hitomezashi pattern on $\Z/7\Z \times \Z/7\Z$ with $x = \lbrace1,1,1,1,-1,-1,1\rbrace$, $y = \lbrace 1,1,-1,-1,1,-1,1\rbrace$.}\label{fig1}
\end{figure}
    
	Oriented hitomezashi patterns with $G, H \in \ssthat{\mathbb{Z}, \mathbb{Z}/2N\Z}{N\in \Z_+}$ decompose into two unoriented hitomezashi patterns in $\Cloth_{G\times H}$; see \cite[Section 1]{RenZhang} for details.
	
Henceforth we will work exclusively with oriented hitomezashi patterns, and after Section 1 we will assume that all hitomezashi patterns, loops, or paths are oriented. We call hitomezashi patterns on grids $G\times H$ \textbf{planar} if $G$ and $H$ are both infinite, \textbf{cylindrical} if one of $G$ and $H$ is infinite and the other is finite, and \textbf{toroidal} if both are finite. We will also refer to the hitomezashi paths and loops on such grids with the same modifiers. Without loss of generality, given a cylindrical hitomezashi loop on $\Cloth_{G\times H}$ we will always assume $G$ is infinite and $H$ is finite.

    We will need the following definitions to help compute homology classes.
\par\noindent
\begin{definition}
		\label{defn:cylinderhomology}
		Let $L$ be a cylindrical hitomezashi loop. The \textbf{homology class} of $L$ is $\Delta y/N$, where $\Delta y$ is the difference of the number of edges in $L$ of the form $(i, j) \rightarrow (i, j+1)$ and the number of edges in $L$ of the form $(i, j+1) \rightarrow (i, j)$. Loops with homology class 0 are called \textbf{trivial}; loops with nonzero homology class are called \textbf{nontrivial}.
	\end{definition}

\begin{definition}
\label{defn:toroidhomology}
Let $L$ be a toroidal hitomezashi loop. The \textbf{homology class} of $L$ is $(\Delta x/M, \Delta y/N)$ where the definition of $\Delta x$ is analogous to that of $\Delta y$.
Loops with homology class 0 are called \textbf{trivial}; loops with nonzero homology class are called \textbf{nontrivial}.
\end{definition}

	We recall some observations made by Ren and Zhang in \cite{RenZhang} about toroidal hitomezashi patterns. We can make analogous observations about cylindrical hitomezashi patterns.
	\begin{observation}
		\label{obs:possiblehomology}
		 On any given toroidal hitomezashi pattern, any two hitomezashi loops have no transverse intersection with each other. Further, all hitomezashi loops has no transverse intersection with itself. Since the intersection number of loops with homology classes $(m, n)$ and $(m', n')$ is $mn' - m'n$, all loops must have homology class of the form $(ku, kv)$ for some integer $k$ and for some fixed $u, v$ coprime. Since nonprimitive loops must have a transverse self-intersection, any nontrivial toroidal hitomezashi loop must have homology $\pm(u, v)$ with $u, v$ coprime. By a similar line of reasoning, nontrivial cylindrical hitomezashi loops must have homology class $\pm 1$.
	\end{observation}
The following simple quantities will allow us to use the encoding strings to predict the appearances of hitomezashi loops.
	\begin{definition}
		\label{defn:Kforline}
		Let $x\in \lbrace -1 , 1\rbrace^{\Z}$. Let $a, b \in \Z$. When $a \leq b$, define
		\[\Sigma_a^b (x) = \sum_{a\leq i\leq b} x_i.\]
        In the interest of respecting the properties
        $\Sigma_a^b (x) + x_{b+1} = \Sigma_a^{b+1}(x)$ and $\Sigma_a^b (x) + x_{a-1} = \Sigma_{a-1}^b(x)$, when $a > b+1$, define $\Sigma_a^b (x) = -\Sigma_{b+1}^{a-1}(x)$ and when $a = b+1$ define $\Sigma_a^b(x) = 0$. Define
		\[\Sigma_{max} (x) = \sup \lbrace \Sigma_a^b (x): a,b\in \Z, a \leq b + 1\rbrace\]
		\[\Sigma_{min} (x) = \inf \lbrace \Sigma_a^b (x): a,b\in \Z, a \leq b + 1\rbrace\]
	\end{definition}
    It will often be useful to lift a cylindrical hitomezashi pattern to a planar hitomezashi pattern. To this end, we define lifts of the binary encoding strings.
	\begin{definition}
		\label{defn:Kforloop}
		Let $x\in \lbrace -1 , 1\rbrace^{\Z/N\Z}$. The \textbf{lift} of $x$ to $\Z$ is the sequence $\tilde{x}\in \lbrace -1, 1\rbrace^\Z$ obtained by mapping indices via the projection $\pi: \Z \to \Z/N\Z$, i.e., \[\tilde{x}_i = x_{\pi(i)}\]
		Analogously define 
		\[\Sigma(x) = \sum_{i\in \Z/N\Z} x_i\]
		\[\Sigma_a^b(x) = \Sigma_a^b(\tilde{x})\]
		\[\Sigma_{max}(x) = \Sigma_{max}(\tilde{x})\]
		\[\Sigma_{min}(x) = \Sigma_{min}(\tilde{x})\] where $\tilde{x}$ is the lift of $x$ to $\Z$.
	\end{definition}
    In the case when $\Sigma(x) = 0$, note that $\Sigma_a^b(x)+\Sigma_{b+1}^{a+CN-1}(x) = 0$ for some sufficiently large constant $C$. Thus $\Sigma_{max}(x) = -\Sigma_{min}(x)$. Further, as $\Sigma_{a-N}^b(x) = \Sigma_a^{b+N}(x) = \Sigma_a^b(x)$, the maximum and minimum must be finite, and obtained by a sequence of length less than $N$. We define the \textbf{range} of $x$, denoted $r(x)$, to be $\Sigma_{max}(x) = -\Sigma_{min}(x)$.
    
        In the planar setting, Pete introduces the notion of the \textbf{enclosing rectangle} of a hitomezashi loop, defined as the smallest region of the form $[a, b] \times [c, d]$ containing the hitomezashi loop, and then considers the pair of subsequences $x_c, x_{c+1}, \ldots, x_d$ and $y_a, y_{a+1}, \ldots, y_b$. We will analogously attribute subsequences to cylindrical hitomezashi loops.
	\begin{definition}
		\label{defn:bounding}
		Let $x\in \lbrace -1, 1\rbrace^{\Z/N\Z}$ and $y\in \lbrace -1, 1\rbrace^\Z$, and suppose that $\Sigma(x) = 0$.\nl
		Let $L$ be a cylindrical hitomezashi loop. The \textbf{bounding cylinder} of $L$ is the smallest cylinder of the form $\ssthat{(x, y)}{m \leq x \leq M}$ containing $L$. The \textbf{bounding sequence} of $L$ is the sequence $y_m, y_{m+1}, \ldots, y_M$.
	\end{definition}

        Notice that trivial cylindrical hitomezashi loops in $\Cloth_{\Z \times \Z/N\Z}(x, y)$ can be interpreted as planar hitomezashi loops in $\Cloth_{\Z \times \Z}(\tilde{x}, y)$. In fact, roughly speaking, nontrivial cylindrical hitomezashi loops occur only when a subsequence accumulates too large a sum for a trivial hitomezashi loop to form.
        
	\begin{definition}    
		\label{defn:overflowing}
		Let $x\in \lbrace -1, 1\rbrace^{\Z/N\Z}$ and $y\in \lbrace -1, 1\rbrace^\Z$, and suppose that $\Sigma(x) = 0$. Define a subsequence $y_m, y_{m+1}, \ldots, y_M$ of $y$ to be \textbf{positively overflowing} if $\Sigma_m^M(y) > r(x)$.
		Further call such a subsequence \textbf{minimally positively overflowing} if, for any indices $m < m' \leq M' < M$, the subsequence $y_{m'}, y_{m'+1},\ldots, y_{M'}$ is not positively overflowing. Define the notions of \textbf{negatively overflowing} and \textbf{minimally negatively overflowing} analogously.
	\end{definition}

	\begin{theorem}
		\label{thm:cylindergenerate}
		Let $x\in \lbrace -1, 1\rbrace^{\Z/N\Z}$ and $y\in \lbrace -1, 1\rbrace^\Z$, and suppose that $\Sigma(x) = 0$.
		For any hitomezashi loop $L$ with homology class $1$ in the hitomezashi pattern $\Cloth_{\Z\times \Z/N\Z}(x,y)$, the bounding sequence of $L$ is minimally positively overflowing. Conversely, if $y_m, y_{m+1},\ldots, y_M$ is minimally positively overflowing, there exists a unique hitomezashi loop $L$ of homology $1$ with bounding sequence $y_m, y_{m+1},\ldots, y_M$. The analogous statements hold for hitomezashi loops of homology class $-1$ and minimally negatively overflowing sequences.
	\end{theorem}
	
	It is natural to ask what happens when $\Sigma(x) \neq 0$. We show that in this regime there is a global obstruction to the existence of nontrivial cylindrical hitomezashi loops.
	\begin{theorem}
		\label{thm:cylinderdegenerate}
		Let  $x\in \lbrace -1, 1\rbrace^{\Z/N\Z}$ and $y\in \lbrace -1, 1\rbrace^\Z$, and suppose that $\Sigma(x) \neq 0$.
		Then no nontrivial cylindrical hitomezashi loops exist in $\Cloth_{\Z\times \Z/N\Z}(x,y)$. Further, there exists an infinite cylindrical hitomezashi path with an edge $(i, c_i) \rightarrow (i+1, c_{i})$ for each $i\in \Z$.
	\end{theorem}
    
	In Section \ref{sec:height}, we briefly cover the bijection Pete \cite{Pete} discovered in the planar case, and describe the notion of ``height", to be used throughout the remainder of the paper. We prove Theorem \ref{thm:cylindergenerate} and Theorem \ref {thm:cylinderdegenerate} in Section \ref{sec:cylinderresults}. Then, in Section \ref{sec:torusresults} we apply our study of cylindrical hitomezashi loops to resolve some questions posed by Ren and Zhang in \cite[Section 5]{RenZhang} about toroidal hitomezashi loops.
	\begin{corollary}
		\label{cor:range}
		Suppose $x\in \lbrace -1, 1\rbrace^{\Z/N\Z}$ and $y\in \lbrace -1, 1\rbrace^{\Z/M\Z}$ and $\Sigma(x) = \Sigma(y) = 0$.
		\begin{itemize}
			\item If $r(x) < r(y)$ then loops with homology class $(0, 1)$ and $(0, -1)$ exist.
			\item If $r(x) = r(y)$ then no nontrivial loops exist.
			\item If $r(x) > r(y)$ then loops with homology class $(1, 0)$ and $(-1, 0)$ exist.
		\end{itemize}
	\end{corollary}
	\section{Height}\label{sec:height}

	% Let $x\in \lbrace -1, 1\rbrace^{\Z/N\Z}$. Observe that $\Sigma_{max}$ and $\Sigma_{min}$ are invariant under cyclic permutations of $x$.
	% Note that $\Sigma_{a-N}^b(x) = \Sigma_a^{b+N}(x) = \Sigma_a^b(x)+\Sigma(x)$.
	% Thus if $\Sigma(x) > 0$, $\Sigma_{max}(x) = \infty$ and if $\Sigma(x) < 0$, $\Sigma_{min}(x) = -\infty$. If $\Sigma(x) = 0$, then both $\Sigma_{max}(x)$ and $\Sigma_{min}(x)$ are finite.
    
	% In fact, if $\Sigma(x) = 0$, then $\Sigma_a^b(x)+\Sigma_{b+1}^{a+CN-1}(x) = 0$ for some sufficiently large constant $C$. Thus $\Sigma_{max}(x) = -\Sigma_{min}(x)$.
	
	% To summarize, for $x\in \lbrace -1, 1\rbrace^{\Z/N\Z}$ we have:
	% \begin{itemize}
	% 	\item If $\Sigma(x) < 0$, then $\Sigma_{min}(x) = -\infty$, $\Sigma_{max}(x) < \infty$
	% 	\item If $\Sigma(x) = 0$, $\Sigma_{max}(x) = -\Sigma_{min}(x)$, and both $\Sigma_{max}(x), \Sigma_{min}(x)$ are finite.
	% 	\item If $\Sigma(x) > 0$, then $\Sigma_{max}(x) = \infty$, $\Sigma_{min}(x) > -\infty$
	% 	\item Whenever $\Sigma_{max}(x)$ (resp. $\Sigma_{min}(x)$) is finite, there exists a sequence of length less than $N$ achieving the supremum (resp. infimum). 
	% \end{itemize}
	% The last item follows by the property $\Sigma_{a-N}^b(x) = \Sigma_a^{b+N}(x) = \Sigma_a^b(x)+\Sigma(x)$.
	% In the case  $\Sigma(x) = 0$ we define the $\textbf{range}$ of $x$, denoted $r(x)$, to be $\Sigma_{max}(x) = -\Sigma_{min}(x)$.
    We first describe the bijection for planar hitomezashi loops. To do this, we define excursions.
        \begin{definition}
		\label{defn:excursion}
		Define a sequence  $x\in \lbrace-1, 1\rbrace^N$ to be an up-excursion (resp. down-excursion) if $\Sigma_1^N(x) = 0$ and $\Sigma_1^k(x)$ is positive (resp. negative) for all $1\leq k\leq N-1$.\nl
		We define the \textbf{height} of an up-excursion to be the maximum value of $\Sigma_1^k(x)$. The \textbf{height} of a down-excursion is the negative of the minimum value of $\Sigma_1^k(x)$.
	\end{definition}
        Removing the first and last steps of an excursion $x$ gives a bijection between up- (resp. down-) excursions of length $N$ and height $h$, and Dyck paths of semilength $\tfrac{N}{2}-1$ and height $h-1$.
	\begin{theorem}[Pete~\cite{Pete}]
		For every hitomezashi loop in $\Cloth_{\Z\times \Z}(x,y)$, the smallest enclosing rectangle $[a,b]\times [c,d]$ has the property that $x_c, \ldots, x_d$, and $y_a, \ldots, y_b$ are opposite excursions of the same height; that is, one is an up-excursion and the other is a down-excursion. Conversely any rectangle $[a,b]\times [c,d]$ with such a property has a unique hitomezashi loop spanning the rectangle.
	\end{theorem}
    
	Pete establishes an equivalent bijection for unoriented hitomezashi patterns by analyzing a \emph{height function} on edges and square regions;  see \cite[Section 3]{Pete}. 
	
        Pete's height function on the unoriented hitomezashi pattern involves an alternating sum. For oriented hitomezashi patterns, the height function is simpler. Since excursions will not appear in the remainder of the paper, all subsequent mentions of \emph{height} will refer to this height function.
	\begin{definition}
		\label{defn:region}
		Let $x, y \in \lbrace -1, 1\rbrace^\Z$. Define a \textbf{square region} to be a region of form $[i, i+1] \times [j, j+1]$ for $i, j \in \Z$. (For brevity, we will always refer to square regions by their center, and sometimes refer to them as simply ``regions".)
        \end{definition}
        \begin{definition}
        \label{defn:heightfunction}
        The \textbf{height} of a square region on $\Cloth_{\Z\times \Z}(x,y)$ with center $(i+\frac{1}{2}, j+\frac{1}{2})$ is $\Sigma_0^j(x)-\Sigma_0^i(y)$. We further define the \textbf{height} of any edge to be the average of the heights of the two regions it borders - that is, the height of an edge with vertices $(i, j), (i + 1, j)$ has the average height of the regions of center $(i + \frac{1}{2}, j + \frac{1}{2})$ and $(i + \frac{1}{2}, j - \frac{1}{2})$, and similarly, the height of an edge with vertices $(i, j), (i, j + 1)$ has the average height of the regions of center $(i - \frac{1}{2}, j + \frac{1}{2})$ and $(i + \frac{1}{2}, j + \frac{1}{2})$.
	\end{definition}

\begin{figure}[ht]
    \centering
    \includegraphics[height=8cm]{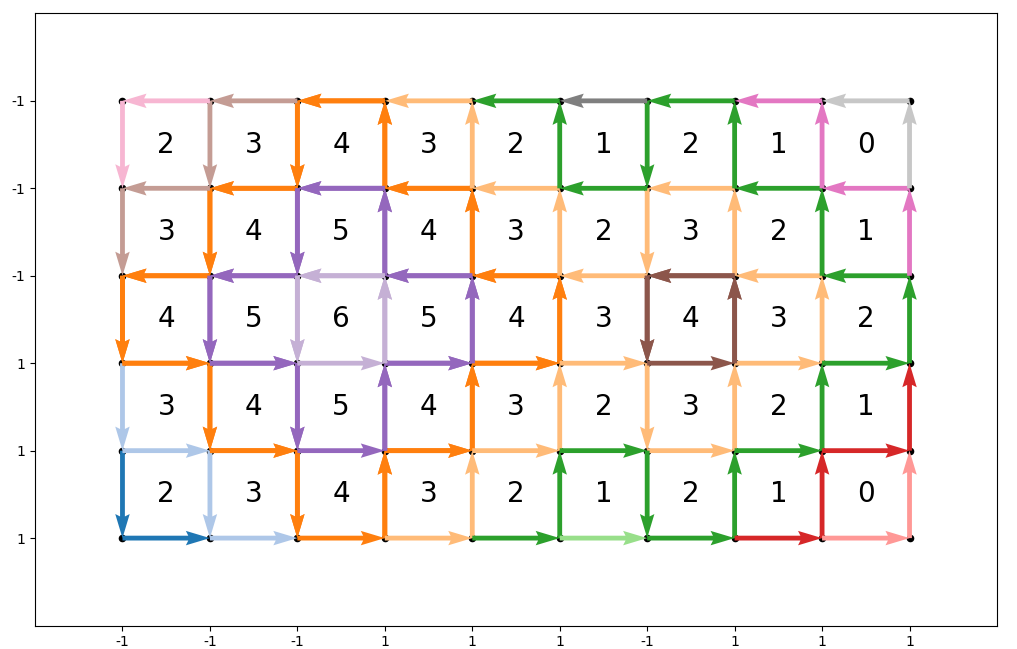}
\caption{A section of a planar hitomezashi pattern with
$x = \{ \dots, 1,1,1,-1,-1,-1, \dots \}$, 
$y = \{ \dots, -1,-1,-1,1,1,1,-1,1,1,1, \dots \}$, with the origin in the bottom left corner. Each square region displays its height.}\label{fig3}
\end{figure}

    When traversing an edge, the height of the region to the left will always be one larger than that of the region to the right. For instance an edge of the form  $(i, j) \rightarrow (i + 1, j)$ implies $x_j = 1$, so the region with center $(i + \frac{1}{2}, j + \frac{1}{2})$ must have height one larger than the region with center $(i + \frac{1}{2}, j - \frac{1}{2})$. Analogous arguments hold for all possible directions of edges.

    Now observe that the regions in Figure \ref{fig3} to the left of the bright orange loop have height $4$, one more than the height of the regions to the right. This is not a coincidence: Take for example the edge from $(3, 1) \rightarrow (4, 1)$. Since the region below has height $3$ the edge implies that the region above must have height $4$. The next edge in the path, $(4, 1) \rightarrow (4, 2)$ implies that the region to its right must have height $3$. Since hitomezashi paths alternate between horizontal and vertical edges, we can always continue in this fashion and conclude that all regions to the left (resp. right) have the same height. Thus all edges on a hitomezashi path have the same height. We will occasionally refer to the height of a hitomezashi path or loop, which is the height of any one of its edges.
    
	Finally, we also define height on  cylindrical and toroidal hitomezashi patterns by lifting to the plane, whenever such a definition would be well-defined.
	\begin{definition}
		\label{defn:heightfunction}
		Let $x\in \lbrace -1, 1\rbrace^{\Z/N\Z}$ and $y\in \lbrace -1, 1\rbrace^\Z$, and suppose $\Sigma(x) = 0$. Then the height of the square region on $\Cloth_{\Z\times \Z/N\Z}(x,y)$ with center $(i+\frac{1}{2}, j+\frac{1}{2})$ is simply the height of the square region on $\Cloth_{\Z\times \Z}(\tilde{x},y)$ with center $(i+CN+\frac{1}{2}, j+\frac{1}{2})$, for an arbitrary integer $C$.
	\end{definition}
        Notice that any choice of $C$ will produce the same height function, as the height of the region with center $(i+C_2N+\frac{1}{2}, j+\frac{1}{2})$ minus the height of the region with center $(i+C_1N+\frac{1}{2}, j+\frac{1}{2})$ is $\Sigma_{i+C_1N+1}^{i+C_2N}(x) = (C_2-C_1)\Sigma(x) = 0$.
        
	We define height for toroidal patterns whenever $\Sigma(x) = \Sigma(y) = 0$ in the same fashion.
	\section{Nontrivial cylindrical hitomezashi loops}\label{sec:cylinderresults}
	\subsection{Characterization when $\Sigma(x) = 0$}
	Recall from Observation \ref{obs:possiblehomology} that the only possible nontrivial homology classes of hitomezashi loops are $\pm 1$. The goal of this section is to characterize such hitomezashi loops by proving Theorem \ref{thm:cylindergenerate}, which we do by first proving Proposition \ref{prop:exists} and Proposition \ref{prop:satisfies}, which correspond to the forwards and backwards directions of Theorem \ref{thm:cylindergenerate}.
	We will only give proofs for minimally positively overflowing subsequences and loops with homology class $1$. The argument is identical for minimally negatively overflowing subsequences and loops with homology class $-1$.
	\begin{proposition}
		\label{prop:exists}
		Suppose $\Cloth_{\Z\times \Z/N\Z}(x,y)$ is a cylindrical hitomezashi pattern with $\Sigma(x) = 0$, and $m < M$ are integers such that $y_m, y_{m+1},\ldots, y_M$ is positively overflowing (resp. negatively overflowing).\nl
		Then there exists a hitomezashi loop with homology class $1$ (resp. $-1$) in the region $\ssthat{(x, y)}{m \leq x \leq M}$.
	\end{proposition}
	\begin{proof}
		It suffices to prove Proposition \ref{prop:exists} when $y_m, y_{m+1},\ldots, y_M$ is minimally positively overflowing. Note that the minimality implies $y_m = y_M = 1$, since otherwise removing the start or end would result in a smaller positively overflowing sequence. Since  $\Sigma(x) = 0$, we may assume without loss of generality that $\Sigma_0^k(x) > 0$ is nonnegative for all $0\leq k \leq N-1$, because cyclically permuting $x$ does not change $\Cloth_{\Z\times \Z/N\Z}(x,y)$ (up to graph isomorphism).

		Since $y_m = 1$ we must have the directed edge $(m, -1)\rightarrow (m,0)$. Let $L$ be the hitomezashi loop or infinite hitomezashi path containing the edge $(m, -1)\rightarrow (m,0)$.
        We will show that $L$ cannot contain any edges of the form $(m, j) \rightarrow (m - 1, j)$ or $(M, j) \rightarrow (M + 1, j)$, confining it to the region $\ssthat{(x, y)}{m \leq x \leq M}$. Since there are a finite number of edges in this region, this implies $L$ is indeed a hitomezashi loop.
        
        We will then show that $L$ cannot contain any edges of the form $(i, 0)\rightarrow (i, -1)$. This implies that $\Delta y > 0$, which then implies that $L$ must have homology class $1$. To see this, let us traverse $L$ starting at $(m, -1)\rightarrow (m,0)$, tracking the number of upward edges minus the number of downward edges. If at any point in our traverse $\Delta y = 0$, then the last edge must have been of the form $(i, 0) \rightarrow (i, -1)$. Thus there are always more upward than downward edges when traversing the loop, and so $\Delta y > 0$. Hence the homology class $\Delta y/N$ of $L$ is also positive, and must be $1$.

        \begin{figure}[ht]
			\centering
            \includegraphics[height = 10cm]{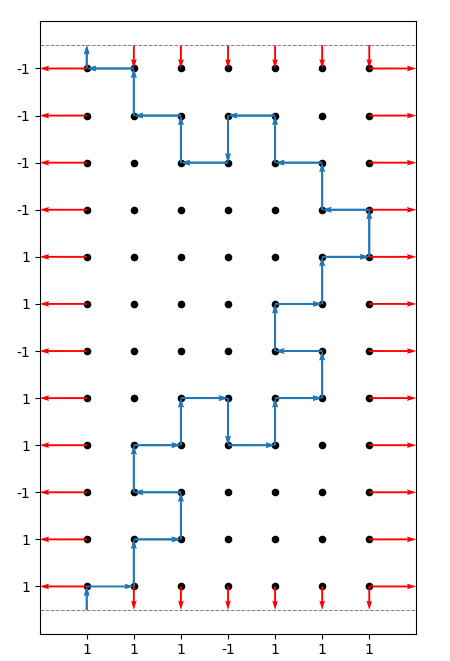}            
			\caption{An illustration of the existence argument. Intuitively, the hitomezashi loop $L$ (in blue) is contained by the horizontal forbidden edges and can only rejoin its edge $(-1, 0) \rightarrow (0, 0)$ by going up the cylinder due to the vertical forbidden edges, as indicated in red.}\label{FigHito2}
		\end{figure}
        
        It remains to show that such edges cannot be in $L$. We demonstrate this by computing the heights of such edges, and showing that they cannot be the same as the height of $L$.
        
        Note that $(m, -1)\rightarrow (m,0)$ borders regions with height $\Sigma_{0}^{-1} (x) - \Sigma_0^{m-1}(y)$ and $\Sigma_{0}^{-1} (x) - \Sigma_0^{m}(y)$. Thus, every edge in $L$ has height $-\frac{1}{2}y_m - \Sigma_0^{m-1}(y) = -\frac{1}{2} - \Sigma_0^{m-1}(y)$.
		
         Consider an edge $(m, j) \rightarrow (m-1, j)$. Such an edge implies $x_j = -1$. This edge borders regions with height $\Sigma_0^{j-1}(x) - \Sigma_0^{m-1}(y)$ and $\Sigma_0^{j}(x) - \Sigma_0^{m-1}(y)$. Thus $(m-1, j) \rightarrow (m, j)$ has height \[\frac{1}{2}x_j+\Sigma_0^{j-1}(x)-\Sigma_0^{m-1}(y) = - \frac{1}{2}+\Sigma_0^{j-1}(x)-\Sigma_0^{m-1}(y) =  \frac{1}{2}+\Sigma_0^{j}(x)-\Sigma_0^{m-1}(y)\] since $(m, j) \rightarrow (m-1, j)$ implies $x_j = -1$. Since $\frac{1}{2}+\Sigma_0^{j}(x)-\Sigma_0^{m-1}(y) \geq \frac{1}{2} -\Sigma_0^{m-1}(y) > - \frac{1}{2} -\Sigma_0^{m-1}(y)$, $L$ cannot contain any edge of the form $(m, j) \rightarrow (m-1, j)$.
		
		Now consider an edge $(M, j) \rightarrow (M+1, j)$. Such an edge implies $x_j = 1$. This edge borders regions with height $\Sigma_0^{j-1}(x) - \Sigma_0^{M}(y)$ and $\Sigma_0^{j}(x) - \Sigma_0^{M}(y)$. Hence $(M, j) \rightarrow (M+1, j)$ has height \[-\frac{1}{2}x_j + \Sigma_0^{j}(x) - \Sigma_0^{M}(y) = -\frac{1}{2} + \Sigma_0^{j}(x) - \Sigma_0^{M}(y)\] As $\Sigma_m^{M}(y) = r(x) + 1$, the height of $(M, j) \rightarrow (M+1, j)$ is $-\frac{1}{2} + \Sigma_0^{j}(x) - \Sigma_0^{M}(y) = -\frac{3}{2} + \Sigma_0^{j}(x) - r(x) - \Sigma_0^{m-1}(y)$. Since $r(x) \geq \Sigma_0^{j}(x)$, $-\frac{3}{2} + \Sigma_0^{j}(x) - r(x) - \Sigma_0^{m-1}(y) \leq -\frac{3}{2}  - \Sigma_0^{m-1}(y) <  -\frac{1}{2}  - \Sigma_0^{m-1}(y)$, and thus $L$ cannot contain any edge $(M, j) \rightarrow (M + 1, j)$.
		
		Finally, consider the edge $(i, 0)\rightarrow (i, -1)$ for $m \leq i \leq M$. Such an edge implies $y_i = -1$. It borders regions with height $\Sigma_0^{-1}(x)-\Sigma_0^i(y)$ and $\Sigma_0^{-1}(x)-\Sigma_0^{i-1}(y)$. Thus, $(i, 0)\rightarrow (i, -1)$ has height 
        \[\frac{1}{2}y_i - \Sigma_0^{i}(y) =\frac{1}{2}y_i - \Sigma_0^{m-1}(y) -\Sigma_m^i(y) = -\frac{1}{2} - \Sigma_0^{m-1}(y) -\Sigma_m^i(y)\]
        Since $\Sigma_m^i(y) > 0$, $-\frac{1}{2} - \Sigma_0^{m-1}(y) -\Sigma_m^i(y) < -\frac{1}{2} - \Sigma_0^{m-1}(y)$, so no such edge is contained in $L$.

        This completes the proof, and therefore $L$ is a hitomezashi loop of homology $1$.
	\end{proof}
	
	\begin{proposition}
		\label{prop:satisfies}
    		Suppose $\Cloth_{\Z\times \Z/N\Z}(x,y)$ is a cylindrical hitomezashi pattern with $\Sigma(x) = 0$, and let $L$ be a hitomezashi loop in $\Cloth_{\Z\times \Z/N\Z}(x,y)$ with homology class $1$ (resp. $-1$). Then the bounding cylinder $\ssthat{(x, y)}{m \leq x \leq M}$ has the property that there exist $m\leq m' < M' \leq M$ such that $y_{m'}, y_{m'+1}, \ldots, y_{M'}$ is minimally positively overflowing (resp. minimally negatively overflowing).
	\end{proposition}
	\begin{proof}
		It suffices to find a positively overflowing subsequence, as then there must exist a minimally positively overflowing subsequence within it. Since $r(x) = \Sigma_{max}(x)$ there exist $a, b$ such that $\Sigma_a^b (x) = r(x)$. Since $\Sigma_{a+CN}^b(x) = \Sigma_a^b(x) - C\Sigma(x) = \Sigma_a^b(x)$, we can choose $0 \leq b - a < N$. Without loss of generality, we may cyclically permute $x$ so that $0 = a \leq b < N$. Since $L$ has homology class $1$, it must contain an upward edge at every $y$–coordinate. In other words, for each $j \in \mathbb{Z}/N\mathbb{Z}$ there exists some $c_j$ with $(c_j, j)\to(c_j, j+1)$ belonging to $L$. Let $m' = c_{-1}$ and $M' = c_b$. Note that since these are the $x$-coordinates of vertices in $L$, $m\leq m' < M' \leq M$, where $\lbrace m \leq x \leq M \rbrace$ is the bounding cylinder of $L$.

		As $L$ contains the edge $(m', -1)\rightarrow (m', 0)$, $L$ has height \[-\frac{1}{2}y_{m'}-\Sigma_{0}^{m-1}(y) = -\frac{1}{2}-\Sigma_{0}^{m-1}(y)\] since $(m', -1)\rightarrow (m', 0)$ implies $y_{m'} = 1$.
        
		Note that the edge $(M', b)\rightarrow (M',b+1)$ has height \[\frac{1}{2}y_{M'}+\Sigma_0^b(x)-\Sigma_0^{M'}(y) = \frac{1}{2}+\Sigma_0^b(x)-\Sigma_0^{M'}(y)\] since $(M', b)\rightarrow (M', b+1)$ implies $y_{M'} = 1$. Since
		$(M', b)\rightarrow (M',b+1)$ is in $L$, 
		\[-\frac{1}{2}-\Sigma_{0}^{m-1}(y) = \frac{1}{2}+\Sigma_0^b(x)-\Sigma_0^{M'}(y)\]
		Rearranging yields $\Sigma_{m'}^{M'}(y) = 1+\Sigma_0^b(x)$. Since $a = 0$ and $\Sigma_a^b(x) = r(x)$, $\Sigma_{m'}^{M'}(y) = r(x)+1$. Thus $y_{m'},y_{m'+1},\ldots, y_{M'}$ is positively overflowing.
	\end{proof}
	
	\begin{proof}[Proof of Theorem~\ref{thm:cylindergenerate}]
		Let $L$ be a hitomezashi loop with homology class 1 and bounding cylinder $\ssthat{(x, y)}{m' \leq x \leq M'}$. Then by Proposition \ref{prop:satisfies} there exists $m \leq m' < M' \leq M$ such that $y_{m'}, y_{m'+1}, \ldots, y_{M'}$ is positively overflowing. By Proposition \ref{prop:exists}, there exists a hitomezashi loop $L'$ with homology class 1 in the cylinder $\ssthat{(x, y)}{m' \leq x \leq M'}$. Suppose for sake of contradiction that $L' \neq L$. Then $L'$ and $L$ are edge-disjoint and have no transverse intersections. Since $L$ has bounding cylinder $\ssthat{(x, y)}{m \leq x \leq M}$, it has vertical edges $u, v$ at $x = m$ and $x = M$. But since $L'$ has homology class 1 it splits the cylinder into two connected components with $v$ and $u$ in distinct components, so there is no path between $u$ and $v$ without transverse intersections with $L'$, which is a contradiction. Therefore, $L' = L$, so $m' = m$ and $M' = M$, implying $y_m, y_{m+1}, \ldots, y_M$ is minimally positively overflowing.
		
		\begin{figure}[ht]
			\centering
			\includegraphics[height=10cm]{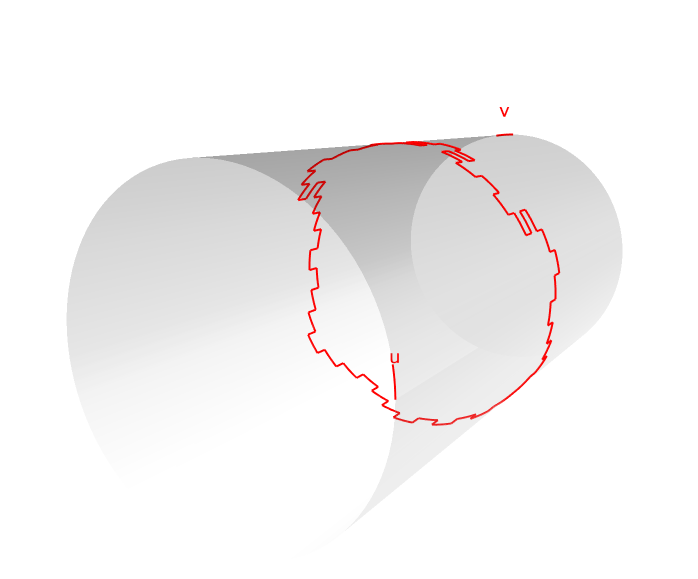}
			\caption{An illustration of the uniqueness argument from the proof of Theorem \ref{thm:cylindergenerate}.
            Here, the loop $L'$ in red disconnects edge $u$ from edge $v$.}\label{FigHito3}
		\end{figure}
		
		Conversely, suppose $y_m, y_{m+1}, \ldots, y_M$ is minimally positively overflowing. Then by Proposition \ref{prop:exists} there exists a hitomezashi loop $L$ with homology class 1 within $\ssthat{(x, y)}{m \leq x \leq M}$. By the above argument, $L$ has bounding cylinder $\lbrace m' \leq x \leq M' \rbrace$ such that $y_{m'}, y_{m'+1}, \ldots, y_{M'}$ is minimally positively overflowing. Thus $m' = m$ and $M' = M$, so $L$ indeed has bounding cylinder $\ssthat{(x, y)}{m \leq x \leq M}$. It remains to show that $L$ is unique. Suppose $L'$ has homology class 1, and has bounding cylinder $\lbrace m \leq x \leq M\rbrace$ and $L\neq L'$. Then by the same argument we can find edges $v, u$ in $L$ such that $L'$ splits the cylinder with $v,u$ in different components, which is a contradiction. Thus, $L = L'$, proving uniqueness. See Figure \ref{FigHito3}.
	\end{proof}
	\subsection{Nonexistence when $\Sigma(x) \neq 0$}
	The purpose of this section is to prove Theorem \ref{thm:cylinderdegenerate}. To do this, we begin with the following observation.
	\begin{lemma}
	\label{lem:deltax}
	%Begin with observation about Delta x
	Consider the hitomezashi pattern $\Cloth_{\Z\times \Z/N\Z}(x,y)$ restricted to edges in $\ssthat{(x, y)}{a \leq x \leq b}$.
	If $\Sigma(x) \neq 0$, then there exists a hitomezashi path with endpoints on $x = a$ and $x = b$.
	\end{lemma}
	\begin{proof}
	Note that the sum of $\Delta x$ over all hitomezashi paths and loops restricted to $\ssthat{(x, y)}{a \leq x \leq b}$ in $\Cloth_{\Z\times \Z/N\Z}$ is $(b-a)\Sigma(x)$, the number of rightward edges minus the number of leftward edges.
	Recall that all hitomezashi loops have $\Delta x = 0$. Each hitomezashi path has endpoints in $\ssthat{(x, y)}{x \in \lbrace a, b \rbrace}$. Hitomezashi paths with both endpoints on $x = a$ or both endpoints on $x = b$ have $\Delta x = 0$, while $\Delta x = b-a$ for ``leftward'' paths from $x = a$ to $x = b$ and $\Delta x = a-b$ for  ``rightward'' paths from $x = b$ to $x = a$. Thus, $(b-a)[\text{\# rightward paths} - \text{\# leftward paths}] = \Sigma(x)(b-a)$
	so
		\[\text{\# rightward paths} - \text{\# leftward paths} = \Sigma(x).\]
	Hence, $\Sigma(x) > 0$ implies the existence of a rightward path from $x = a$ to $x = b$, whereas $\Sigma(x) < 0$ implies the existence of a leftward path from $x = b$ to $x = a$. Either way, such a leftward or rightward hitomezashi path has endpoints on $x = a$ and $x = b$.
	\end{proof}
	Now it remains to show Lemma \ref{lem:deltax} implies Theorem \ref{thm:cylinderdegenerate}.
	\begin{proof}[Proof of Theorem \ref{thm:cylinderdegenerate}]
	Suppose for sake of contradiction that some hitomezashi loop $L$ with homology class $\pm 1$ existed in $\Cloth_{\Z\times \Z/N\Z}(x,y)$. Then it must have some bounding cylinder $\lbrace m\leq x \leq M \rbrace$. But by Lemma \ref{lem:deltax} there exist vertices $u, v$ connected by a hitomezashi path such that $u$ is on the line $x = m-1$ and $v$ is on the line $x = M+1$. This path must transversely intersect $L$, which is a contradiction. Thus, no such nontrivial hitomezashi loop exists.
 	
 	Now suppose that there does not exist an infinite hitomezashi path with edges of the form $(i, c_i) \rightarrow (i + 1, c_i)$ exists in $\Cloth_{\Z\times \Z/N\Z}$. Then all edges must either be part of a trivial hitomezashi loop, or an infinite hitomezashi path that is contained in a subset of the form $\ssthat{(x, y)}{x \leq M}$ or $\ssthat{(x, y)}{m \leq x}$. We will show that there must exist $N+1$ (and in fact, infinitely many) hitomezashi loops or infinite hitomezashi paths containing an edge of the form $\lbrace(0, j), (1, j)\rbrace$.
    
	Let $0 = a_0<b_0 = 1$. By Lemma \ref{lem:deltax} there exists a hitomezashi path $P_0$ with endpoints on $x = a_0$ and $x = b_0$. By assumption, $P_0$  must be part of either a finite hitomezashi loop $L_0$, which has some bounding cylinder $\ssthat{(x, y)}{m_0 \leq x \leq M_0}$ or an infinite hitomezashi loop contained in $\ssthat{(x, y)}{m_0 \leq x}$ and $\ssthat{(x, y)}{x \leq M_0}$.     
    
        We use the following protocol for constructing $L_{n+1}$ from $L_n$. 
            
        In the case where $L_n$ is finite and has bounding cylinder $\ssthat{(x, y)}{m_n \leq x \leq M_n}$ choose $a_{n+1} = m_n - 1, b_{n+1} = M_n + 1$. In the case where $L_n$ is infinite and contained in $\ssthat{(x, y)}{m_n \leq x}$, choose $a_{n+1} = m_n - 1, b_{n+1} = b_{n+1}$. In the case where $L_n$ is infinite and contained in $\ssthat{(x, y)}{x \leq M_n}$, choose $a_{n+1} = a_{n+1}, b_{n+1} = M_n + 1$.

        Let $P_{n+1}$ be a hitomezashi path connecting $a_{n+1}$ and $b_{n+1}$. Note that $P_{n}$ cannot be part of $L_k$ for $0 \leq k \leq n$, as no $L_k$ contains both vertices at $x = a_{n+1}$ and $x = b_{n+1}$.
        Hence $P_n$ is part of either a finite hitomezashi path or an infinite hitomezashi path $L_{n+1}$ that is distinct from $L_0, L_1, \ldots, L_n$. In such a manner, we construct $N+1$ hitomezashi paths $L_0, \ldots, L_{N+1}$, each of which are part of a distinct finite hitomezashi loop or infinite hitomezashi path. Thus the hitomezashi paths $L_0, \ldots, L_N$ are pairwise edge-disjoint. But since $P_k\subseteq L_k$ is a path with endpoints $x = a_k$ to $x=b_k$, and $a_k \leq 0 = a_0 < b_0 = 1\leq b_k$, each of $P_0, \ldots, P_N$ contains an edge $\lbrace (0,j), (1, j)\rbrace$.\nl

 	\begin{figure}[ht]
	\centering
	\includegraphics[height=10cm]{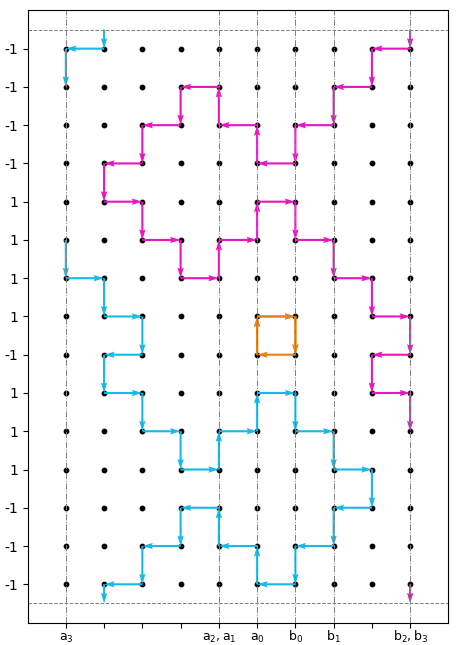}
	\caption{Examples of all three cases in the proof of Theorem \ref{thm:cylinderdegenerate}. $L_0$ is the finite loop in brown, $L_1$ is the infinite path with a rightmost edge, indicated in blue, and $L_2$ is the infinite path with a leftmost edge, in purple.}\label{FigHito4}
	\end{figure}
	
	Since there are only $N$ such edges, this is a contradiction and there must exist an infinite hitomezashi path.
	
	\end{proof}
	\section{Homology class of hitomezashi loops on the torus}\label{sec:torusresults}
In toroidal hitomezashi patterns with $\Sigma(x) \neq 0$ and $\Sigma(y) \neq 0$, Ren and Zhang \cite{RenZhang} determined both the homology classes and the number of nontrivial hitomezashi loops. Building on this, we give a complete description of the nontrivial loops that arise in toroidal hitomezashi patterns. In Subsection \ref{subsec:torclass} we construct a bijection, analogous to Theorem \ref{thm:cylindergenerate}, for the case when either $\Sigma(x) = 0$ or $\Sigma(y) = 0$, and use it to resolve an open question posed by Ren and Zhang. In Subsection \ref{subsec:torcount} we apply this bijection to determine the homology classes and counts of nontrivial hitomezashi loops in all toroidal patterns, thereby completing their results in this direction.
	\subsection{Possible homology classes}\label{subsec:torclass}
	%Repeat ren and zhang's topological observations about possible homology classes
	%Describe their results, and explain how our corollary resolves something in theirs

    We recall Ren and Zhang's results about the possible homology classes of toroidal hitomezashi loops. 
	\begin{theorem}[Ren-Zhang ~\cite{RenZhang}]
	Suppose $L$ is a nontrivial toroidal hitomezashi loop in $\Cloth_{\Z/M\Z \times \Z/N\Z}(x,y)$. Then $L$ must have one of the following homology classes, as described in the table below. 
    \[
    \resizebox{0.9\textwidth}{!}{%
    \begin{tabular}{|c|c|c|}
    	\hline
    	&$\Sigma(x)\neq 0$&$\Sigma(x) = 0$\\
    	\hline
    	$\Sigma(y)\neq 0$&
    	$\Sigma(x)/\gcd(|\Sigma(x)|,|\Sigma(y)|), \Sigma(y)/\gcd(|\Sigma(x)|,|\Sigma(y)|))$
    	&$(0,\pm1)$\\
    	\hline
    	$k(y) = 0$&$(\pm1,0)$&$(\pm1,0)$ or $(0,\pm1)$\\
    	\hline
    \end{tabular}%
    }
    \]
	\end{theorem}
	Recall from Observation $\ref{obs:possiblehomology}$ that on any given toroidal hitomezashi pattern $\Cloth_{\Z\times \Z/N\Z}(x,y)$, all nontrivial toroidal hitomezashi loops must be of the form $\pm (u,v)$ with $u, v$ coprime. Thus, when $\Sigma(x) = \Sigma(y) = 0$, loops of the form $\pm (0, 1)$ and $\pm (1,0)$ cannot coexist. In \cite[Problem 5.1]{RenZhang} Ren and Zhang ask:
    \begin{problem}\label{problem:range}
        When $\Sigma(x) = \Sigma(y) = 0$, is there a simple criterion on $x, y$ that tells us which homology classes can exist?
    \end{problem}
    \begin{figure}[ht]
    \includegraphics[height = 8cm]{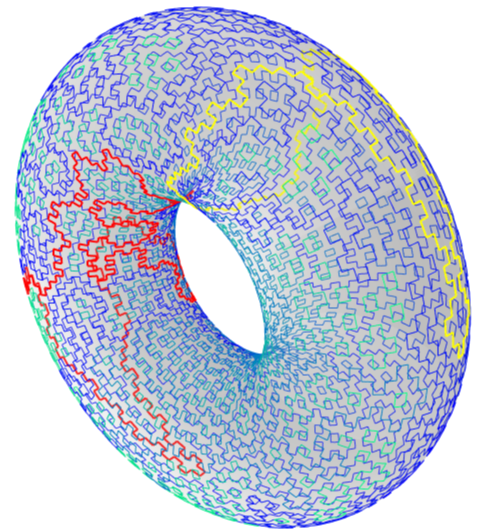}
    \caption{An example of a large (unoriented) toroidal hitomezashi pattern, with $\Sigma(x) = \Sigma(y) = 0$. The prominent red and yellow hitomezashi loops are the only ones with nonzero homology class.}\label{fig:toroidal}
    \end{figure}
    We answer this question in Corollary $\ref{cor:range}$.

	To show this, let us first observe an analogous bijection to Theorem $\ref{thm:cylindergenerate}$ by lifting to the cylinder.
	\begin{definition}
		\label{def:arc}
		Let $y \in \lbrace -1, 1\rbrace^{\Z/M\Z}$.
		An \textbf{arc} of $y$ is a contiguous sequence $y_a, y_{a+1}, \ldots, y_b$ of elements of $y$.
	\end{definition}
	Note that an arc of $y$ may not necessarily be a subsequence of $y$: the sequence may have more than $M$ elements and may loop around $y$ any number of times. Also, note that $y_a, y_{a+1}, \ldots, y_b$ and $y_{a+M}, y_{a+M+1}, \ldots, y_{b+M}$ are the same arc, but $y_a, y_{a+1}, \ldots, y_b$ and $y_{a}, y_{a+1}, \ldots, y_b, y_{b+1}, \ldots, y_{b+M}$ are different arcs.
	\begin{definition}
		Let $L$ be a hitomezashi loop in $\Cloth_{\Z/M\Z \times \Z/N\Z}(x,y)$ of homology $(0, \pm 1)$ (resp. ($\pm 1, 0$)). Then the \textbf{bounding arc} of $L$ is the bounding sequence of any of the preimages of $L$ when lifting to the cylindrical hitomezashi pattern $\Cloth_{\Z \times \Z/N\Z}(x,\tilde{y})$ (resp. $\Cloth_{\Z/M\Z \times \Z}(\tilde{x}, y)$).
	\end{definition}
	Note that this notion is well-defined since any two preimages of the same loop will be identical up to a horizontal (resp. vertical) shift by some multiple of $M$ (resp. $N$). Thus, their bounding arcs are the same.
	\begin{corollary}
		\label{cor:toroidal}
		Let $x\in \lbrace -1, 1\rbrace^{\Z/N\Z}$ such that $\Sigma(x) = 0$, and let $y\in \lbrace -1, 1\rbrace^{\Z/M\Z}$. Then, if $L$ is a hitomezashi loop of homology $(0, 1)$ (resp. $(0, -1)$), its bounding arc is a minimally positively overflowing arc (resp. a minimally negatively overflowing arc). Conversely, if there is an arc $y_a, \ldots, y_b$ that is minimally positively overflowing (resp. minimally negatively overflowing) then there exists a unique hitomezashi loop $L$ of homology $(0, 1)$ (resp. $(0, -1)$) with bounding arc $y_a, \ldots, y_b$.
	\end{corollary}
	Of course, the analogous statement holds on toroidal hitomezashi patterns of the form $\Cloth_{\Z/M\Z\times \Z/N\Z}(x,y)$ when $\Sigma(y) = 0$. 
	\begin{proof}
		We prove Corollary \ref{cor:toroidal} for minimally positively overflowing arcs and loops of homology $(0, 1)$; the argument for minimally negatively overflowing arcs and loops with homology class $(0, -1)$ is identical. Begin by lifting $\Cloth_{\Z/M\Z \times \Z/N\Z}(x, y)$ to the cylindrical pattern $\Cloth_{\Z \times \Z/N\Z}(x,\tilde{y})$.
        
        Then loops with homology class $(0, 1)$ in the toroidal pattern correspond to infinitely many loops with homology class $1$, all identical up to horizontal shifting by $M$. By Theorem \ref{thm:cylindergenerate}, loops with homology class $1$ on the cylinder correspond to minimally positively overflowing subsequences of $\tilde{y}$. Certainly the set of such subsequences is invariant under horizontal shift by $M$. Say two subsequences of $\tilde{y}$ are equivalent if they are identical up to a horizontal shift by some multiple of $M$. Similarly, two hitomezashi loops are equivalent if one can be obtained from the other via a horizontal shift by some multiple of $M$.
        
        It is clear that each equivalence class of minimally positively overflowing subsequences in $\tilde{y}$ corresponds to a single minimally positively overflowing arc. Also, each equivalence class of minimally positively overflowing subsequences corresponds to an equivalence class of hitomezashi loops with homology class $1$, which corresponds to a single hitomezashi loop of homology $(0,1)$ in $\Cloth_{\Z/M\Z \times \Z/N\Z}$. Thus minimally positively overflowing arcs are in bijection with  hitomezashi loops of homology $(0,1)$, as desired.
	\end{proof}
	Now we answer Problem \ref{problem:range}.
	\begin{proof}[Proof of Corollary \ref{cor:range}]
		Suppose $r(x) < r(y)$. Then by definition there exists an arc whose sum is greater than $r(x)$. Thus, this arc is a positively overflowing arc, which implies there exists a minimally positively overflowing arc. Thus, there exists a hitomezashi loop $L$ of homology $(0, 1)$. The argument is analogous when $r(x) > r(y)$. Finally, suppose $r(x) = r(y)$. Then there do not exist any overflowing arcs (with respect to either $x$ or $y$). Thus, no loops with homology class $(0, \pm 1)$ or $(\pm 1, 0)$ exist.
	\end{proof}
	
	\subsection{Nontrivial loop count}\label{subsec:torcount}
	%yippity yip yappity
	%yappity yip yippity
	In this section, we collect some results about the number of nontrivial hitomezashi loops, furthering the work of Ren and Zhang.
	\begin{theorem}[Ren, Zhang ~\cite{RenZhang}]
		Let $\Cloth_{\Z/M\Z \times \Z/N\Z}(x,y)$ be a toroidal hitomezashi pattern with $\Sigma(x) \neq 0$ and $\Sigma(y)\neq 0$. The number of nontrivial hitomezashi loops is $\gcd(|\Sigma(x)|,|\Sigma(y)|)$.
	\end{theorem}
	Using our discussion in the previous section, we can complete the classification of the number of nontrivial hitomezashi loops in the cases when one or both of $\Sigma(x)$ and $\Sigma(y)$ are zero.
	\begin{corollary}
		\label{cor:toroidalcount}
		Let $\Cloth_{\Z/M\Z \times \Z/N\Z}(x,y)$ be a toroidal hitomezashi pattern. Define $c(x,n)$ to be the number of minimal arcs of $x$ whose sum is more than $n$ if $n>0$ and the number whose sum is less than $n$ if $n<0$. We will also use the notation $\sgn(x)$ to be the sign function on $\Sigma(x)$. Then the number and homology class of the hitomezashi loops present are as follows: 
		\begin{itemize}
			\item If $\Sigma(x)\neq 0$ and $\Sigma(y) \neq 0$, then there are $\gcd(|\Sigma(x)|,|\Sigma(y)|)$ loops with homology class $(\Sigma(x)/\gcd(|\Sigma(x)|,|\Sigma(y)|),\Sigma(y)/\gcd(|\Sigma(x)|,|\Sigma(y)|))$.
			\item If $\Sigma(x) = 0$ and $\Sigma(y)\neq 0$, then there are $c(y,-r(x)\sgn(y))$ loops with homology class $(0, -\sgn(y))$ and $|\Sigma(y)|+c(y,-r(x)\sgn(y))$ loops with homology class $(0, \sgn(y))$.
			\item If $\Sigma(x) \neq 0$ and $\Sigma(y) = 0$, then there are $c(x,-r(y)\sgn(x))$ loops with homology class $(-\sgn(x), 0)$ and $|\Sigma(x)|+c(x,-r(y)\sgn(x))$ loops with homology class $(\sgn(x),0)$.
			\item If $\Sigma(x) = \Sigma(y) = 0$ then we have the following cases:
			\begin{itemize}
				\item If $r(x) > r(y)$, then there are $c(y,r(x))$ loops with homology class $(1,0)$ and $c(y,r(x))$ loops with homology class $(-1, 0)$.
				\item If $r(x) = r(y)$, then there are no nontrivial loops of any homology class.
				\item If $r(x) < r(y)$, then there are $ c(x,r(y))$ loops with homology class $(0,1)$ and $c(x,r(y))$ loops with homology class $(0, -1)$.
			\end{itemize}
		\end{itemize}
	\end{corollary}

        Observe that in a hitomezashi pattern $\Cloth_{\Z/M\Z \times \Z/N\Z}(x,y)$, each $x_j\in x$ determines the orientation of $N$ edges with vertices $((i, j), (i+1, j))$. Thus
        \[\text{\# rightward edges} - \text{\# leftward edges} = N\Sigma(x)\]
        Analogously
        \[\text{\# upward edges} - \text{\# downward edges} = M\Sigma(y)\]
        Let $\textbf{L}$ be the set of all hitomezashi loops in $\Cloth_{\Z/M\Z \times \Z/N\Z}(x,y)$. Since every edge is contained in exactly one hitomezashi loop, 
        \[\sum_{L \in \textbf{L}}(\Delta x/N, \Delta y/M) = (\Sigma(x), \Sigma(y))\]

         Thus the sum of the homology classes of hitomezashi loops in $\Cloth_{\Z/M\Z \times \Z/N\Z}(x,y)$ is $(\Sigma(x), \Sigma(y))$.

         We are now ready to prove Corollary \ref{cor:toroidalcount}.
	\begin{proof}
	The case where $\Sigma(x) \neq 0$ and $\Sigma(y)\neq 0$ is proven by Ren and Zhang in \cite[Section 4]{RenZhang}. The other cases follow immediately from Corollary \ref{cor:toroidal}.
         
	In the case where $\Sigma(x) = 0$ and $\Sigma(y)\neq 0$, we have that $c(y, r(x))$ is the number of minimally positively overflowing arcs of $y$, and $c(y, -r(x))$ is the number of minimally negatively overflowing arcs. Since the only possible nontrivial hitomezashi loops have homology class $\pm (0, 1)$ and the homology classes of hitomezashi loops must sum to zero, we must have $c(y, r(x)) - c(y, -r(x)) = \Sigma(x)$.
    
    If $\Sigma(y) > 0$, then any minimally negatively overflowing arc of $y$ must be a subset of $y$, making $c(y, -r(x))$ easy to compute. If $\Sigma(y) < 0$, then the same is true of $c(y, r(x))$.  Regardless, there must be $c(x, -\sgn(y)r(x))$ loops with homology class $(0, -\sgn(y))$.There must be $|\Sigma(y)|$ more loops with homology class $(0, \sgn(y))$ than with homology class $(0, -\sgn(y))$. Thus there must be $|\Sigma(y)| + c(x, -\sgn(y)r(x))$ loops with homology class $(0, \sgn(y))$.
	
	Of course the case when $\Sigma(x)\neq 0$ and $\Sigma(y) = 0$ follows from an identical argument.
	
	In the case when $\Sigma(x) = \Sigma(y) = 0$, as shown in Corollary \ref{cor:range} if $r(x) = r(y)$ then no nontrivial hitomezashi loops exist. If $r(x) < r(y)$ we get vertical hitomezashi loops. Since $\Sigma(y) = 0$ it is clear that the number of loops with homology class $(0, 1)$ and homology class $(0, -1)$ is equal. Thus they are both equal to $c(y, r(x))$. The case where $r(x) > r(y)$ follows from an identical argument.
	\end{proof}
	
	\section*{Acknowledgements}
	This research was conducted at the University of Minnesota Duluth REU with support from NSF Grant No. DMS-2409861, Jane Street Capital, and donations from Ray Sidney and Eric Wepsic. We thank Noah Kravitz, Mitchell Lee, Maya Sankar, and Tuong Le for helpful discussions during the research process. We thank Mitchell Lee and Eliot Hodges for their feedback during the editing process. Finally, we thank Joe Gallian and Colin Defant for their support, and the opportunity to participate in the Duluth REU.

\nocite{*}
\bibliographystyle{plain}

\end{document}